\documentclass[12pt]{article}

\usepackage{amsmath, amssymb, amsfonts, amsthm, array, enumerate, float, lineno, mathrsfs, subcaption, tikz,fullpage}

\usepackage{caption}
\usepackage{subcaption}

\usepackage{forest}

\usepackage{multirow}
\usepackage{hhline}

\usepackage{geometry}
\renewcommand{\S}{\mathcal{S}}

\newcommand{\qbar}{\overline{\mathcal{Q}}}
\newcommand{\QQ}{\overline{\mathcal{Q}}}
\newcommand{\q}{\overline{q}}

\DeclareMathOperator{\des}{des}
\DeclareMathOperator{\asc}{asc}
\DeclareMathOperator{\pl}{pl}
\DeclareMathOperator{\Set}{Set}

\newcommand{\set}[3]{#1 = \{#2 : #3\}}

\newcommand{\abs}[1]{\left|#1\right|}
\renewcommand{\S}{\mathcal{S}}
\newcommand{\Q}{\mathcal{Q}}
\renewcommand{\set}[1]{\left\{#1\right\}}
\newcommand{\T}{\mathcal{T}}

\newtheorem{proposition}{Proposition}
\newtheorem{lemma}{Lemma}
\newtheorem{theorem}{Theorem}
\newtheorem{corollary}{Corollary}
\theoremstyle{definition}

\newtheorem{example}{Example}

\title{Pattern restricted quasi-Stirling permutations}
\author{Kassie Archer, Adam Gregory, Bryan Pennington, and Stephanie Slayden}
\date{}

\begin{document}

\maketitle

\begin{abstract}
We define a variation of Stirling permutations, called quasi-Stirling permutations, to be permutations on the multiset $\{1,1,2,2,\ldots, n,n\}$ that avoid the patterns 1212 and 2121. Their study is motivated by a known relationship between Stirling permutations and increasing ordered rooted labeled trees. We construct a bijection between quasi-Stirling permutations and the set of ordered rooted labeled trees and investigate pattern avoidance for these permutations.
\end{abstract}

\section{Introduction}

Stirling permutations were introduced in 1978 by Gessel and Stanley \cite{GS1978} in the course of their investigation of Stirling polynomials. These polynomials have the 
form $f_k(n) = S(n+k, n)$ and $g_k(n) = c(n,n-k)$, where $S(n,m)$ and $c(n,m)$ denote the Stirling numbers of the second kind and signless Stirling numbers of the 
first kind, respectively. The set of Stirling permutations, typically denoted by $\Q_n$, consists of permutations on the multiset $\{1,1,2,2,\ldots, n,n\}$ that avoid the 
pattern 212. That is, if $\pi=\pi_1\pi_2\ldots \pi_{2n} \in \Q_n$ and $i<j<k$ with $\pi_i=\pi_k$, then we must have $\pi_j>\pi_i$. In \cite{GS1978}, Gessel and Stanley 
enumerated these permutations and found that $|\Q_n| = (2n-1)!!$. They subsequently refined this enumeration by the number of descents and found that the resulting numbers
$B_{n,i}$ satisfy Eulerian-like polynomials in $f_k(n)$ and $g_k(n)$. In particular, if $B_{n,i}$ denotes the number of Stirling permutations of length $2n$ with exactly 
$i$ descents, then 
\[
	\sum_{n=0}^\infty f_k(n) x^n = \left(\sum_{i=1}^k B_{k,i}x^i\right)/(1-x)^{2k+1}
\]
and 
\[
	\sum_{n=0}^\infty g_k(n) x^n = \left(\sum_{i=k+1}^{2k} B_{2k-i+1,i}x^i\right)/(1-x)^{2k+1}.
\]

Stirling permutations have received a considerable amount of attention since their conception. They have been found to be in bijection with several interesting 
combinatorial objects \cite{J2008,JKP2011,MY2015,  DY2014} (including ordered rooted labeled trees, which we discuss further in this paper); several permutation statistics on them have been studied, 
including descents, ascents, plateaus, and left-to-right maxima and minima \cite{GS1978, KP2011, B2008, JW2015}; and several applications of these permutations have been discovered. 
In addition, pattern avoidance on Stirling permutations was explored in \cite{KP2012, JW2015,CMM2016}.
\begin{figure}[h]
\centering
\forestset{
    my tree/.style={
      for tree={
        shape=circle,
        fill=black,
        minimum width=5pt,
        inner sep=1pt,
        anchor=center,
        line width=1pt,
        s sep+=25pt,
      },
    },
  }
  \begin{forest}
    my tree
    [, 
      [, label={left:$1$}
      [, label = left:$2$
      [, label = left:$5$]
      [, label = left:$7$]
      [, label = right:$9$]]
      [, label={right:$10$}]]
      [, label={right:$3$}
        [, label={below:$4$}]
        [, label={right:$6$}
          [, label={right:$8$}]
        ]]
    ]
  \end{forest}
  \caption{Consider the increasing ordered rooted tree pictured here. Traversing counterclockwise, we obtain the Stirling permutation $\pi = 125577992(10)(10)134468863$.}
   \label{fig:inctree}
\end{figure}
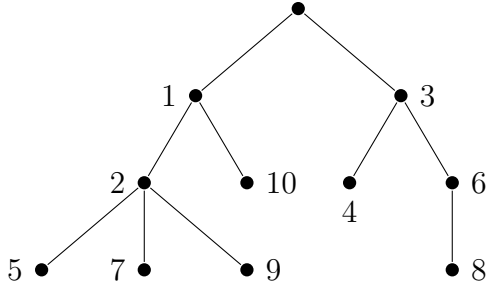

In this paper, the authors define quasi-Stirling permutations. This definition is motivated by a characterization of Stirling permutations 
via a bijection with increasing ordered rooted trees with labels in $[n]$ (see for example \cite{J2008}). Given an increasing ordered rooted tree on $[n]$ as seen in Figure \ref{fig:inctree}, one can obtain 
a Stirling permutation by performing a depth-first walk, i.e., by traversing the tree counterclockwise and recording the vertex below each edge traversed. Since each edge 
must be traversed twice, one obtains a permutation on the multiset $[n,n]$. Since this tree is increasing (and so no descendant of a given vertex is less than the vertex 
itself), this permutation will avoid the pattern $212$. This map is invertible and thus this defines a bijection. 

In this paper, we extend this bijection to all  ordered rooted labeled trees. In Section~\ref{sec: background}, we introduce the necessary background. In Section~\ref{sec: trees}, we describe the bijection between  rooted labeled ordered trees and quasi-Stirling permutations. In Section~\ref{sec: enumeration}, we present the bulk of our enumerative results. Namely, this section includes the enumeration of quasi-Stirling permutations that avoid all sets of size at least 2 containing patterns of length 3. We also enumerate these permutations with respect to plateaus. Finally, in Section~\ref{sec: open}, we present a few open questions and directions for future research.

\section{Background}\label{sec: background}

\subsection{Permutation statistics and pattern avoidance}
We denote the set of all permutations on the multiset $[n,n]=\{1,1,2,2,\ldots, n,n\}$ by $\S_{n,n}$. For a permutation $\pi \in \S_{n,n}$ and for $i<j$, we will denote by $\pi_{[i,j]}$ the segment $\pi_i\pi_{i+1} \ldots \pi_j$. For a given segment $\pi_{[i,j]}$, we denote by $\Set(\pi_{[i,j]})$ the subset of $[n]$ given by:
$$
\Set(\pi_{[i,j]}) = \{\pi_i, \pi_{i+1}, \ldots, \pi_j\}
$$
without multiplicities.
For example, $\Set(144288) = \{1,2,4,8\}$.
For $\pi \in \S_{n,n}$ and $i<j$, we say that the segment $\pi_{[i,j]}$ is \textbf{increasing} if $\pi_i\leq \pi_{i+1}\leq \cdots\leq \pi_j$ and similarly, we say that the segment is \textbf{decreasing} if $\pi_i\geq \pi_{i+1}\geq \cdots\geq \pi_j$. 

For a permutation $\pi=\pi_1\pi_2\ldots \pi_{2n} \in \S_{n,n}$, we say $i$ is 
a \textbf{descent} of $\pi$ if $\pi_i>\pi_{i+1}$, and we denote the number of descents of $\pi$ by $\des(\pi)$. Similarly, we say $i$ is an \textbf{ascent} of $\pi$ if 
$\pi_i<\pi_{i+1}$, and we denote the number of ascents of $\pi$ by $\asc(\pi)$. Finally, we say that $i$ is a \textbf{plateau} of $\pi$ if $\pi_i=\pi_{i+1}$ and denote the number 
of plateaus by $\pl(\pi)$. For example, if $\pi = 77611632554423\in S_{7,7}$, then $\des(\pi) = 6$, $\asc(\pi) = 3$, and $\pl(\pi) = 4$. Notice that for any $\pi \in \S_{n,n}$, 
we must have $\des(\pi) + \asc(\pi) + \pl(\pi) = 2n-1$. 

Pattern avoidance is a well-studied concept for traditional permutations (see \cite{SS1985} for the seminal paper on the topic and \cite{Kitaevbook} for a survey) as well as for generalizations of permutations and other combinatorial objects. For sequences $\sigma$ and $\tau$, we say $\sigma$ and $\tau$ are in the same relative order if (1) $\sigma_i=\sigma_j$ if and only if $\tau_i=\tau_j$, and (2) $\sigma_i<\sigma_j$ if and only if $\tau_i<\tau_j$. 
We say that a permutation $\pi=\pi_1\pi_2\ldots\pi_{2n}\in\S_{n,n}$ \textbf{contains} the pattern $\sigma=\sigma_1\sigma_2\ldots\sigma_k$ if there is some 
$i_1<i_2<\cdots<i_k$ so that $\pi_{i_1}\pi_{i_2}\ldots\pi_{i_k}$ is in the same relative order as $\sigma$. If $\pi$ does not contain $\sigma$, we say that $\pi$ 
\textbf{avoids} $\sigma$.  We denote by $\S_{n,n}(\sigma_1, \sigma_2, \ldots,\sigma_m)$ the set of permutations in $\S_{n,n}$ that avoid all of the patterns 
$\sigma_1, \sigma_2, \ldots,\sigma_m$. 

The set $\Q_n$ of \textbf{Stirling permutations} consists of the permutations in $\S_{n,n}$ that avoid the pattern 212. That 
is, $\Q_n:=S_{n,n}(212)$ and thus there is no $i<j<k$ so that $\pi_i=\pi_k$ with $\pi_j<\pi_i$. For example, $135532214664$ is a Stirling permutation since there is no subsequence in the same relative order as 212. However, the 
permutation $4225541331$ is not a Stirling permutation since it contains the subsequence $424$, which is an occurrence of the pattern 212. The permutation 
$4225541331$ does avoid the pattern $1212$ and also avoids $123$.

We define the set of \textbf{quasi-Stirling permutations} to be the set $\QQ_n := \S_{n,n}(1212, 2121)$, i.e. those permutations in $\S_{n,n}$ that avoid both 1212 and 
2121. For example, $4225541331$ is a quasi-Stirling permutation, but $4221554331$ is not since the subsequence 4141 is an occurrence of the pattern 2121. We will denote the 
set of quasi-Stirling permutations that additionally avoid the patterns in the set $\Lambda = \{\sigma_1, \sigma_2, \ldots,\sigma_m\}$ by $\QQ_n(\sigma_1, \sigma_2, \ldots,\sigma_m)$ or just $\QQ_n(\Lambda)$. We use the notation $\q_n(\sigma_1, \sigma_2, \ldots, \sigma_m)$ or $\q_n(\Lambda)$ to denote the size of the set $\qbar_n(\Lambda)$, i.e.
$$ \q_n(\Lambda)=  |\QQ_n(\Lambda)|.$$
Finally, we say that two patterns $\sigma$ and $\tau$ are \textbf{$\boldsymbol\QQ$-Wilf-equivalent} if $\q_n(\sigma) = \q_n(\tau)$ for all $n\geq 1$ and we say that two sets 
$\Lambda_1 = \{\sigma_1, \sigma_2, \ldots, \sigma_k\}$ and $\Lambda_2=\{\tau_1, \tau_2, \ldots, \tau_m\}$ are $\QQ$-Wilf-equivalent if $\q_n(\Lambda_1) 
= \q_n(\Lambda_2)$ for all $n\geq 1$. 

\subsection{Ordered rooted trees}

We define a \textbf{rooted labeled tree} on $[n]=\{1,2,\ldots, n\}$ to be a tree on $n+1$ nodes, $n$ of which are labeled with the numbers from $[n]$. The lone vertex that remains 
unlabeled is called the \textbf{root}. We draw these trees as in Figure~\ref{fig:rooted tree} with the root at the top.

\begin{figure}[h]
\centering
\forestset{
    my tree/.style={
      for tree={
        shape=circle,
        fill=black,
        minimum width=5pt,
        inner sep=1pt,
        anchor=center,
        line width=1pt,
        s sep+=25pt,
      },
    },
  }
  \begin{forest}
    my tree
    [, 
 [, label={left:$9$}
      [,label = left:$1$
        [, label={left:$8$}]]
        [, label={right:$12$}
          [, label={left:$6$}]
           [, label={right:$10$}]
        ]]
      [, label={right:$2$}
      [, label = left:$5$
      [, label = left:$3$]
      ]
      [, label = left:$4$]
      [, label = left:$11$
      [, label={left:$7$}]]]]
       \end{forest}
\forestset{
    my tree/.style={
      for tree={
        shape=circle,
        fill=black,
        minimum width=5pt,
        inner sep=1pt,
        anchor=center,
        line width=1pt,
        s sep+=25pt,
      },
    },
  }
  \begin{forest}
    my tree
    [, 
      [, label={left:$2$}
      [, label = left:$4$]
      [, label = left:$5$
      [, label = left:$3$]]
      [, label = left:$11$
      [, label={left:$7$}]]]
      [, label={right:$9$}
      [,label = left:$1$
        [, label={left:$8$}]]
        [, label={right:$12$}
          [, label={left:$10$}]
           [, label={right:$6$}]
        ]]
    ]
  \end{forest}
  \caption{Two examples of rooted labeled trees on $[12]$. These trees are different as ordered rooted labeled trees. }
  \label{fig:rooted tree}
  \end{figure}
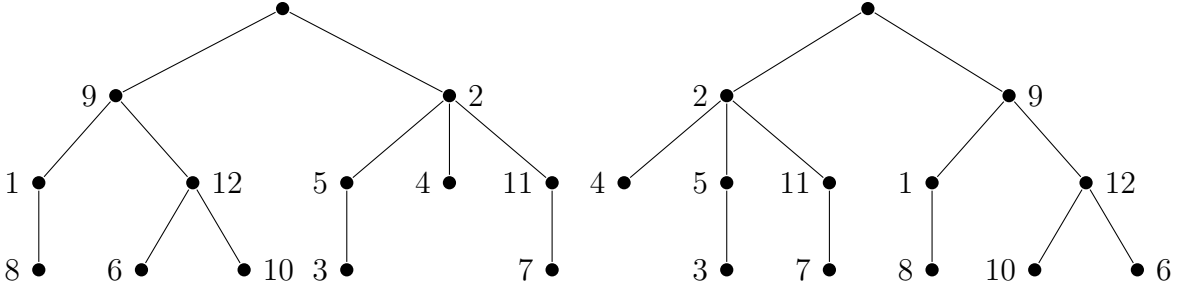

A vertex $j$ is a \textbf{descendant} of vertex $i$ if $i\neq j$ and $i$ appears in the unique path from the root to $j$. In this case, we say that $i$ is an \textbf{ancestor} of $j$. 
If $i$ and $j$ are adjacent in the forest, then we say that $i$ is the \textbf{parent} of $j$ and that $j$ is a \textbf{child} of $i$. For example, in Figure~\ref{fig:rooted tree}, vertices 
2 and 11 are ancestors of 7, while vertex 8 is a child of 1 and a descendant of 9. 
We say that a rooted labeled tree is \textbf{ordered} if the left-to-right order in which the children 
of each vertex appear matters. In other words, different embeddings of the tree into the plane determine distinct ordered rooted labeled trees. We can see that the two trees in 
Figure \ref{fig:rooted tree} are distinct as ordered rooted labeled trees since the children of several vertices appear in different orders in the given embedding. We let $\T_n$ 
denote the set of ordered rooted labeled trees on $[n]$.

\section{Quasi-Stirling permutations and rooted trees}\label{sec: trees}

The authors discovered quasi-Stirling permutations in the process of generalizing the map between increasing rooted labeled trees and Stirling permutations described in \cite{StanleyBook}. Such permutations arose naturally when investigating what happens if the domain of the map is extended to include all ordered rooted trees on $[n]$ (i.e. removing 
the condition that they must be increasing). We describe the resulting map here. 

Denote by $\varphi: \T_n \to \S_{n,n}$ the map which sends an ordered rooted labeled tree to a permutation of the multiset $[n,n]$ obtained by 
traversing the tree counterclockwise and recording each vertex twice. More precisely, we perform a depth-first walk, and every time an edge between parent $i$ and child $j$ is 
traversed, we record $j$. See Example \ref{example varphi} for a detailed example of this map. After the example, we prove that this map actually is a bijection between ordered rooted 
labeled trees and the set of quasi-Stirling permutations, $\QQ_n$, and describe its inverse (demonstrated by Example \ref{varphi inverse example}).

\begin{figure}[h]
\centering
\forestset{
    my tree/.style={
      for tree={
        shape=circle,
        fill=black,
        minimum width=5pt,
        inner sep=1pt,
        anchor=center,
        line width=1pt,
        s sep+=25pt,
      },
    },
  }

  \begin{forest}
    my tree
    [, 
 [, label={3}
      [,label = left:$2$] 
        [, label={left:$6$}]
        [, label={right:$4$}]]
      [, label={right:$5$} 
      [, label = left:$1$]]]
       \end{forest}
	\caption{The rooted labeled ordered tree $T$ associated to the permutation $\pi=\varphi(T) =  32266445115\in \S_{6,6}$. Since this permutation avoids 1212 and 2121, it is indeed a quasi-Stirling permutation.}
		\label{fig: tree example}
	\end{figure}
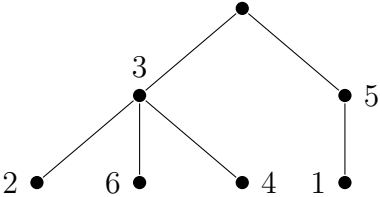

  	\begin{figure}[H]
	\centering
		\begin{tabular}{ccccc}
			\subcaptionbox{Steps 1--3}{
			\begin{tikzpicture}[node distance=1.25cm]
			\node[inner sep=0pt](r) {};
			\node[inner sep=0pt,label=120:3,below left of=r](3) {};
			\node[inner sep=0pt,label=120:2,below left of=3](2) {};
			\node[inner sep=0pt,label=0:4,below right of=3](4) {};
			\node[inner sep=0pt,label=30:5,below right of=r](5) {};
			\node[inner sep=0pt,label=270:6,below of=3](6) {};
			\node[inner sep=0pt,label=0:1,below of=5](1) {};
			\node(ar) {};
			\node[below left of=r](a3) {};
			\node[below left of=3](a2) {};
			\node[below right of=3](a4) {};
			\node[below right of=r](a5) {};
			\node[below of=3](a6) {};
			\node[below of=5](a1) {};
			\draw(2) -- (3) -- (6);
			\draw (4) -- (3) -- (r) -- (5) -- (1);
			\foreach \pt in {r,1,2,3,4,5,6}
				\fill (\pt) circle (2.5pt);
			\draw[->](ar) to [out=180,in=90,looseness=1] node[above] {\scriptsize{1}} (a3);
			\draw[->](a3) to [out=180,in=90,looseness=1] node[above] {\scriptsize{2}} (a2);
			\draw[->](a2) to [out=0,in=250,looseness=1] node[below] {\scriptsize{3}} (a3);
		\end{tikzpicture} } & \hspace{.5cm} &
			\subcaptionbox{Steps 4--6}{
			\begin{tikzpicture}[node distance=1.25cm]
			\node[inner sep=0pt](r) {};
			\node[inner sep=0pt,label=120:3,below left of=r](3) {};
			\node[inner sep=0pt,label=120:2,below left of=3](2) {};
			\node[inner sep=0pt,label=0:4,below right of=3](4) {};
			\node[inner sep=0pt,label=30:5,below right of=r](5) {};
			\node[inner sep=0pt,label=270:6,below of=3](6) {};
			\node[inner sep=0pt,label=0:1,below of=5](1) {};
			\node(ar) {};
			\node[below left of=r](a3) {};
			\node[below left of=3](a2) {};
			\node[below right of=3](a4) {};
			\node[below right of=r](a5) {};
			\node[below of=3](a6) {};
			\node[below of=5](a1) {};
			\draw(2) -- (3) -- (6);
			\draw (4) -- (3) -- (r) -- (5) -- (1);
			\foreach \pt in {r,1,2,3,4,5,6}
				\fill (\pt) circle (2.5pt);
			\draw[->](a3) to [out=250,in=120,looseness=1] node[left] {\scriptsize{4}} (a6);
			\draw[->](a6) to [out=60,in=300,looseness=1]
			node[below right,yshift=-2pt,xshift=-2pt] {\scriptsize{5}} (a3) to
			[out=300,in=170,looseness=1] (a4);
			\draw[->](a4) to [out=90,in=0,looseness=1]
			node[above right, yshift=-2pt,xshift=-2pt] {\scriptsize{6}} (a3) to
			[out=0,in=270,looseness=1] (ar);
		\end{tikzpicture} }
		& \hspace{.5cm} &
			\subcaptionbox{Steps 7--10}{
			\begin{tikzpicture}[node distance=1.25cm]
			\node[inner sep=0pt](r) {};
			\node[inner sep=0pt,label=120:3,below left of=r](3) {};
			\node[inner sep=0pt,label=120:2,below left of=3](2) {};
			\node[inner sep=0pt,label=0:4,below right of=3](4) {};
			\node[inner sep=0pt,label=30:5,below right of=r](5) {};
			\node[inner sep=0pt,label=270:6,below of=3](6) {};
			\node[inner sep=0pt,label=0:1,below of=5](1) {};
			\node(ar) {};
			\node[below left of=r](a3) {};
			\node[below left of=3](a2) {};
			\node[below right of=3](a4) {};
			\node[below right of=r](a5) {};
			\node[below of=3](a6) {};
			\node[below of=5](a1) {};
			\draw(2) -- (3) -- (6);
			\draw (4) -- (3) -- (r) -- (5) -- (1);
			\foreach \pt in {r,1,2,3,4,5,6}
				\fill (\pt) circle (2.5pt);
			\draw[->] (ar) to [out=270,in=180,looseness=1] node[below left] {\scriptsize{7}}(5);
			\draw[->](a5) to [out=240,in=120,looseness=1] node[above left] {\scriptsize{8}} (a1);
			\draw[->](a1) to [out=60,in=300,looseness=1] node[right] {\scriptsize{9}} (a5);
			\draw[->](a5) to [out=90,in=0,looseness=1] node[above right] {\scriptsize{10}} (ar);
		\end{tikzpicture}
 }
		
		\end{tabular}
		\caption{Step-by-step depiction of the tree-traversal in computing $\varphi(T)$ from Example \ref{example varphi}.}
		\label{fig: tree example step-by-step}
	\end{figure}
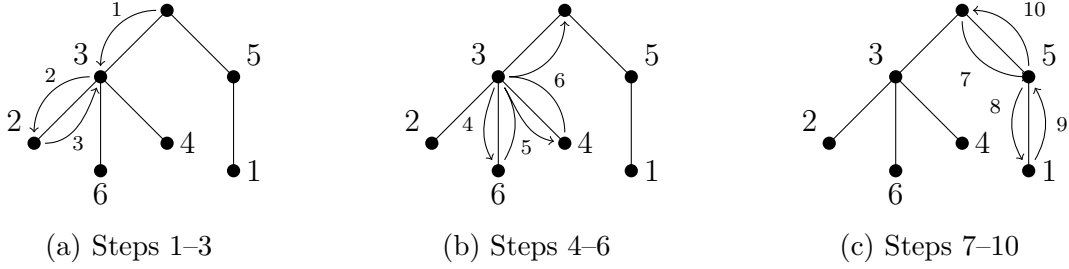

\begin{example}\label{example varphi} Consider the rooted labeled ordered tree $T$ in Figure~\ref{fig: tree example}. Let us construct the quasi-Stirling permutation $\pi=\varphi(T)$. A visual depiction of this step-by-step process is given in Figure~\ref{fig: tree example step-by-step}.
		
		\begin{enumerate}\itemsep=0em
			\item Begin by traversing to the left-most child node of the root.
			Our first value in the permutation is $\pi_1 = 3$.
			\item Since $3$ has children, we continue our traversal to its left-most child $2$.
			We record its value giving us $\pi_{[1,2]} = 32$.
			\item The node labeled 2 has no children, so we traverse back up the tree to $3$. We record the node associated to the edge we are traversing giving us a second $2$. We have thus far obtained $\pi_{[1,3]} = 322$.
			\item Node $3$ still has children that have not yet been visited, so we traverse the edge to the left-most unvisited child node $6$, and record its
			value. Thus $\pi_{[1,4]}  = 3226$.
			\item Traverse back up to $3$ and down to the last remaining child node of 3.
			Record the value 6 as we move up the tree to 3 and record the value 4 as we move down to the child node 4. We have $\pi_{[1,6]}  = 322664$.
			\item Now that all of the children of $3$ have been visited, we
			can record $3$ for the second time and move up to the root, leaving us with $\pi_{[1,8]}  = 32266443$.
			\item Next, we traverse to the next unvisited
			child node of the root. This is our first time visiting this node;
			we obtain $\pi_{[1,9]}  = 322664435$.
			\item We traverse to the only child of $5$, and record its value to get $\pi_{[1,10]}  = 3226644351$.
			\item We travel up towards the root to 5 and record the leaf 1 once again. This gives us $\pi_{[1,11]}  = 32266443511$.
			\item All children of $5$ have been visited. We move up to the root and record 5 to get $\pi_{[1,12]}  = 322664435115$. Since there are no more child nodes of the root to visit, our process is complete and $\pi = 322664435115.$
		\end{enumerate}
\end{example}

	The next theorem asserts that this map is indeed a bijection that sends rooted ordered trees to quasi-Stirling permutations. We do this by first showing that the image of $\varphi$ is contained in $\QQ_n$ and then by constructing an inverse to $\varphi$. 
	
\begin{theorem}
\label{thm:treebij}
The map $\varphi$ described above is a bijection between $\T_n$ and $\QQ_n$. 
\end{theorem}

\begin{proof}
Given any ordered rooted labeled tree $T\in\T_n$, $\varphi(T)$ is clearly an element of $\S_{n,n}$. To show that $\varphi(T) \in \QQ_n$, we must show that this permutation avoids $1212$ and $2121$. If $j$ is a descendant of $i$, $i$ must 
be read first while traversing the tree and $j$ must be read twice before $i$ is read again. If $i$ and $j$ are not related (that is, neither one is a descendant of the other) and $i$ 
appears to the left of $j$ in the tree, then $i$ must be read twice before $j$ is ever read while performing this traversal of the tree. Thus the pattern $abab$ never appears, and in particular $1212$ and $2121$ never appear,
so $\pi$ avoids these patterns. 

It remains to show that the map is invertible on the set $\QQ_n$ of quasi-Stirling permutations. Let $a,b \in [n]$. Consider that an occurrence of a plateau $aa$ indicates a leaf with label $a$; an occurrence of $ab$ where $a\neq b$ and $a$ is appearing for the first time indicates that $a$ is the parent of $b$; an occurrence of $ab$ where $a\neq b$ and $a$ and $b$ are both appearing for the second time indicates that $a$ is a child of $b$; and an occurrence of $ab$ where $a\neq b$, $a$ is appearing for the second time, and $b$ is appearing for the first time indicates that $a$ and $b$ share a parent and $a$ lies directly to the left of $b$ in the embedding of the tree. 

Taken altogether, we can construct the tree $\varphi^{-1}(\pi)$ for any permutation $\pi \in \QQ_n$. First, we want to partition $\pi$ into $m$ blocks where the first
	element of each block is equal to the last. Thus
	\[
		\pi = [a_1 b_1 \cdots c_1 a_1][a_2 b_2 \cdots c_2 a_2]\cdots
		[a_mb_m\cdots c_m a_m].
	\]
	Since $\pi$ avoids $1212$ and $2121$, each block has either exactly 0 or 2 copies of each element in $[n]$.
	The root of $\varphi^{-1}(\pi)$ will have $m$ child nodes labeled from left-to-right
	$a_1,a_2,\ldots,a_m$.
	Continue this process recursively for each block to obtain the subtree associated to each child of the root. A block of the form $aa$ produces a leaf.
\end{proof}

In Example~\ref{varphi inverse example}, we demonstrate the inverse of $\varphi$. 

  	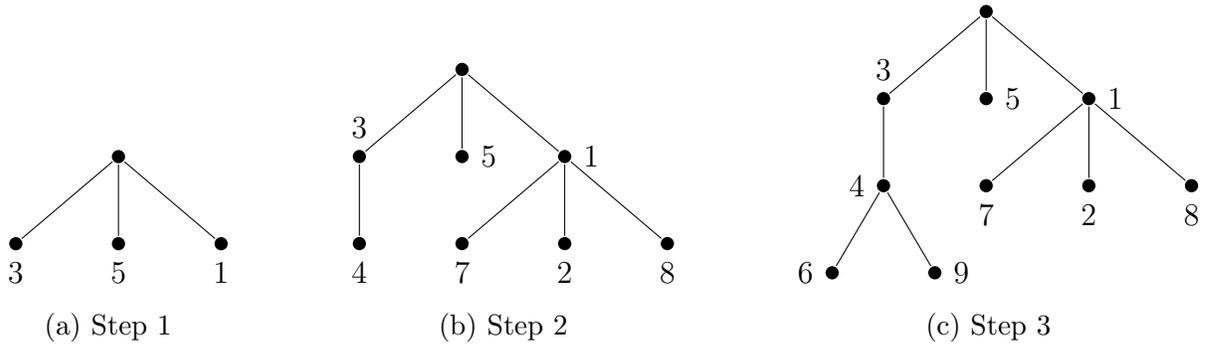
\begin{figure}[H]
	\centering
		\begin{tabular}{ccccc}
			\subcaptionbox{Step 1}{
			\forestset{
    my tree/.style={
      for tree={
        shape=circle,
        fill=black,
        minimum width=5pt,
        inner sep=1pt,
        anchor=center,
        line width=1pt,
        s sep+=25pt,
      },
    },
  }

  \begin{forest}
    my tree
    [, 
 [, label={below:3}]
      [, label={below:$5$}] 
      [, label = {below:$1$}]]
       \end{forest}
 } & \hspace{.1cm} &
			\subcaptionbox{Step 2}{
			\forestset{
    my tree/.style={
      for tree={
        shape=circle,
        fill=black,
        minimum width=5pt,
        inner sep=1pt,
        anchor=center,
        line width=1pt,
        s sep+=25pt,
      },
    },
  }

  \begin{forest}
    my tree
    [, 
 [, label={3}
      [,label = {below:$4$}] ]
      [, label={right:$5$}] 
      [,label={right:$1$}
      [,label={below:$7$}]
            [,label={below:$2$}]
      [, label = {below:$8$}]]]
       \end{forest}
}
		& \hspace{.1cm} &
			\subcaptionbox{Step 3}{
			\forestset{
    my tree/.style={
      for tree={
        shape=circle,
        fill=black,
        minimum width=5pt,
        inner sep=1pt,
        anchor=center,
        line width=1pt,
        s sep+=25pt,
      },
    },
  }

  \begin{forest}
    my tree
    [, 
 [, label={$3$}
      [,label = {left:$4$}
      [,label={left:$6$}]
      [,label={right:$9$}]
      ] ]
      [, label={right:$5$}] 
      [,label={right:$1$}
      [,label={below:$7$}]
            [,label={below:$2$}]
      [, label = {below:$8$}]]]
       \end{forest}			      
			 }
		\end{tabular}
		\caption{Step-by-step depiction of the tree-traversal in computing $\varphi(T)$ from Example \ref{varphi inverse example}.}
		\label{fig: inverse tree example step-by-step}
	\end{figure} 

\begin{example}\label{varphi inverse example}
	Given a quasi-Stirling permutation, $\pi = 346699435517722881$, we will construct the tree $T = \varphi^{-1}(\pi)$. A visual depiction of these three steps can be found in Figure \ref{fig: inverse tree example step-by-step}.
\begin{enumerate}
\item  Partition our permutation $\pi$ into subpermutations such that the
	first element of each subpermutation equals the last element of the
	subpermutation.
	\[
		\pi = [34669943][55][17722881]
	\]
	In this example $3, 5,$ and $1$ are the child nodes of the
	root of $T = \varphi^{-1}(\pi)$.

\item Next, partition each permutation of the children of our root, using the
	same rule in step one: first element is equal to the last element.
	Since the consecutive 5's indicate a plateau, we know that $5$ is a leaf in $T$ and
	that branch of the tree is complete.
	\[
		\pi = [{\color{gray}3}[466994]{\color{gray}3}][{\color{gray}55}][{\color{gray}1}[77][22][88]{\color{gray}1}]
	\]
	Therefore $4$ is the only child node of $3$, but node $1$ has three
	children: $7,2,$ and $8$.
	Also note that $7, 2,$ and $8$ are leafs, so that subtree is now complete.

\item  Now we can partition our last permutation, giving us the child
	nodes of $4$, which are $6$ and $9$.
	\[
		\pi = [{\color{gray}3}[{\color{gray}4}[66][99]{\color{gray}4}]{\color{gray}3}][{\color{gray}55}][{\color{gray}17722881}]
	\]
	Since the consecutive 6's and 9's indicate plateaus, we know our
	process is complete, and a quick traversal of the tree shows that we have
	the correct tree for our permutation.
\end{enumerate}
The final result can be seen in Figure \ref{fig: inverse tree example step-by-step}(c). 
\end{example}


Finally, we can use Theorem \ref{thm:treebij} to enumerate quasi-Stirling permutations. Since it is well known that unlabeled ordered rooted trees on $[n]$ are enumerated by the Catalan numbers (see for example \cite{StanleyBook}), we obtain an immediate 
enumeration of the quasi-Stirling permutations.

\begin{corollary}
For any $n\geq 1$, $|\QQ_n| = n!\cdot C_n$, where $C_n$ is the $n$th Catalan number. 
\end{corollary}

It is clear from this result that quasi-Stirling permutations are in fact in bijection with several other combinatorial objects as well.

\begin{figure}[h]
\begin{center}
\renewcommand{\arraystretch}{1.1}
\begin{tabular}{|c|c|c|}\hline
Representative $\Lambda$ & Enumeration of $\QQ_n(\Lambda)$ & Theorem\\ \hline
$\{123,321\}$ & $0$, for $n \geq 5$ & Theorem \ref{123, 321} \\ \hline
$\{132,312\}$ & \multirow{3}{*}{$4 \cdot 3^{n-2}$, for $n \geq 2$}& Theorem \ref{132, 312} \\ \cline{1-1} \cline{3-3}
$\{132,231\}$ & & Theorem  \ref{132, 231}   \\ \cline{1-1} \cline{3-3}
$\{312,321\}$ & & Theorem  \ref{123, 132}  \\ \hline
$\{132,213\}$ & g.f.: $ \mathcal{A}(x) = \dfrac{(1-x)^2}{x^3 - 3x + 1}$ \rule{0pt}{.7cm} \rule[-.4cm]{0pt}{0pt}& Theorem \ref{132, 213}\\ \hline
$\{132,321\}$ & $2n^2 - 3n + 2$  & Theorem  \ref{123, 312} \\ \hline
$\{123,132,321\}$ & $0$, for  $n \geq 5$ & Theorem  \ref{123, 132, 321}  \\ \hline
$\{132,213,321\}$ & \multirow{3}{*}{$2n$,  for $n \geq 2$} & Theorem  \ref{132, 213, 321}  \\ \cline{1-1} \cline{3-3}
$\{123,213,312\}$ & & Theorem  \ref{123, 213, 312}\\ \cline{1-1} \cline{3-3}
$\{132,213,312\}$ & & Theorem \ref{132, 213, 312} \\ \hline
$\{123,132,213\}$ & $\frac{1}{4} \left( (1+\sqrt{2})^{n+1} + (1-\sqrt{2})^{n+1} - 2 \cdot (-1)^{n+1} \right)$ \rule{0pt}{.4cm} \rule[-.3cm]{0pt}{0pt} & Theorem \ref{123, 132, 213} \\ \hline
$\{123,132, 312\}$ & \rule{0pt}{.4cm} $ \frac{1}{2}(n^2+3n-2)$  \rule[-.3cm]{0pt}{0pt} & Theorem  \ref{123, 132, 312}  \\ \hline
$\{123,132,213,321\}$ &  \multirow{3}{*}{$0$, for $n\geq 5$} & \multirow{3}{*}{Theorem  \ref{123, 132, 213, 321 and more} } \\ \cline{1-1} 
$\{123,132,231,321\}$ &  & \\ \cline{1-1} 
$\{123,132,312,321\}$ & & \\ \hline
$\{123,132,213,231\}$ & $4$, for $n \geq 2$ & Theorem \ref{123, 132, 213, 231} \\ \hline
$\{123,132,231,312\}$ & $3$, for $n \geq 3$ & Theorem \ref{123, 132, 231, 312} \\ \hline
$\{132,213,231,312\}$ & $2$, for $n \geq 3$
      & Theorem \ref{132, 213, 231, 312} \\ \hline
$\{123,132,213,231,321\}$ &$0$, for $n \geq 5$ & Theorem \ref{123, 132, 213, 231, 321} \\ \hline
$\{123,132,213,231,312\}$ & $1$, for  $n \geq 3$
      & Theorem \ref{123, 132, 213, 231, 312}  \\ \hline
\end{tabular}
\end{center}
\caption{A representative of each Wilf-equivalence class is listed and the set of quasi-Stirling permutations that avoid that set of patterns is enumerated.}
\label{table}
\end{figure}

\section{Enumeration results}\label{sec: enumeration}

In this section, we consider the set of quasi-Stirling permutations that avoid a given set of patterns of length three. 
Using two trivial symmetries we can reduce the number of cases we need to consider. The \textbf{reverse} of the permutation $\pi= \pi_1\pi_2\ldots\pi_{2n}\in \S_{n,n}$ is the permutation $\pi^r= \pi_{2n}\pi_{2n-1} \ldots \pi_2\pi_1$ and the
\textbf{complement} of the permutation $\pi\in\S_{n,n}$ is the permutation $\pi^c$ where $\pi^c_i = n+1-\pi_i$. For example, if $\pi = 25513443661277$, then $\pi^r = 77216634431552$ and $\pi^c = 63375445227611$. Clearly, for any 
$\pi\in\S_{n,n}$,  $\pi$ avoids $\sigma$ if and only if $\pi^r$ avoids 
$\sigma^r$ and $\pi^c$ avoids $\sigma^c$. It is also clear that the set of quasi-Stirling permutations is closed under both of these operations. Thus, in this section, for each possible proper nontrivial subset $\Lambda\subseteq \S_3$, we need only consider subsets $\Lambda$ which cannot be obtained via complements or reverses of another subset on our list.


In the table in Figure \ref{table}, we list a representative (of the set of subsets obtained by the reverse and complement symmetries) for sets of size two through five of patterns in $\S_3$. This table 
provides a summary of the results on pattern avoidance in quasi-Stirling permutations that the authors have obtained.

In certain cases, we will need to make use of the following lemma. 
\begin{lemma}\label{no123and321}
For $n\geq 5$, no permutation $\pi \in \S_{n,n}$ avoids both 123 and 321. 
\end{lemma}
Clearly this lemma holds, just as it does in $\S_n$. Indeed, given any subsequence of length five containing five distinct numbers, there must be an occurrence of either 123 or 321.

\subsection{Avoiding one pattern in $\S_3$}

In the case of single patterns, we can deduce that the patterns $123$ and $321$ are trivially $\QQ$-Wilf-equivalent since $123=(321)^r$. Similarly, $213$, $312$, $132$, and $231$ are all $\QQ$-Wilf-equivalent since $213=(312)^r=(231)^c=(132)^{rc}$.  It remains an open question to enumerate quasi-Stirling permutations that avoid a single pattern of length three. In the table in Figure~\ref{fig:avoid one}, we compute $\q_n(132)$ and $\q_n(123)$ for $n\in[5]$. 
\begin{figure}[H]
\centering
\begin{tabular}{|c|c|c|c|c|c|}
\hline
$n$ & 1 & 2 & 3 & 4& 5 \\ \hline
$\q_n(132)$ & 1& 4& 19& 102& 590 \\ \hline
$\q_n(123)$ & 1& 4& 19& 96& 510 \\ \hline
\end{tabular}
\caption{The values $\q_n(132)$ and $\q_n(123)$ for $n\in[5]$.}
\label{fig:avoid one}
\end{figure}

\subsection{Avoiding two patterns in $\S_3$}

\begin{theorem} 
\label{123, 321}
For all $n \geq 5$, $\q_n(123,321) =  0$.
\end{theorem}
\begin{proof}
    This follows immediately from Lemma \ref{no123and321}.
    \end{proof}
    
    \begin{theorem} \label{132, 312}
For all $n \geq 2$, $\q_n(132, 312) = 4 \cdot 3^{n-2}$.
\end{theorem}
\begin{proof}
Let $\Lambda=\{132, 312\}$ and $\pi \in \qbar_n(\Lambda)$. Notice that we must have $\pi_{2n} = n$ or $\pi_{2n}=1$. Indeed, if this were not the case, then for any $i,j<2n$ such that $\pi_i=1$ and $\pi_j=n$, we would have that either that $i<j$, in which case $\pi_i\pi_j\pi_{2n}$ is an occurrence of 132, or that $i>j$, in which case $\pi_j\pi_i\pi_{2n}$ is an occurrence of 312. Thus we have two cases:
\begin{enumerate}[(1)]
\item $\pi=\alpha 1\beta1$, or 
\item $\pi=\alpha n \beta n$. 
\end{enumerate}
In Case (1), $\beta$ is an increasing segment since $\pi$ avoids 132 and for each $a\in \alpha$ and $b\in\beta$, we must have that $a<b$ since $\pi$ avoids 312. In Case (2), $\beta$ is a decreasing segment since $\pi$ avoids 312 and for each $a\in \alpha$ and $b\in\beta$, we must have that $a>b$ since $\pi$ avoids 132. Therefore, for $n\geq 2$,
$$\q_n(\Lambda) = 2\sum_{i=0}^{n-1} \q_n(\Lambda).$$ Solving this recurrence with the initial conditions $q_0(\Lambda) = 1$ and $q_1(\Lambda) = 1$, we obtain the result.
\end{proof}

\begin{theorem} 
\label{132, 231}
For all $n \geq 2$, $\q_n(132, 231) = 4 \cdot 3^{n-2}$.
\end{theorem}
\begin{proof}
    Let $\Lambda = \{132, 231\}$. For a given quasi-Stirling permutation $\pi \in \QQ_{n}(\Lambda)$ with $n>2$, $\pi$ must satisfy one of the following:
    \begin{enumerate}[(1)]
    \item $\pi=nn\alpha$, 
    \item$\pi=\alpha nn$, or 
    \item $\pi=n\alpha n$. 
    \end{enumerate}
    Indeed, if we have some $i<j<k$ so that $\pi_i\neq n$, $\pi_j=n$, and $\pi_k\neq n$, then $\pi_i\pi_j\pi_k$ would be an occurrence of 132 or 231. In each case, $\alpha$ is any permutation in $\QQ_{n-1}(\Lambda)$ since $n$ cannot be part of a 132 or a 231 pattern. We therefore have that for $n>2$, 
    $$\q_n(\Lambda) =3\q_{n-1}(\Lambda).$$
    Since there are four quasi-Stirling permutations in $\QQ_2(\Lambda)$, the result follows.
    \end{proof}
    
    \begin{theorem}
    \label{123, 132}
    For all $n \geq 2$, $\q_n(312, 321) = 4 \cdot 3^{n-2}$.
    \end{theorem}
    \begin{proof}
     Let $\Lambda = \{312, 321\}$ and $n>2$. Write $\pi \in\QQ_n(\Lambda)$ as $\pi=\alpha n \beta n \gamma$. Since $\pi$ avoids 312 and 321, we must have that $|\Set(\beta) \cup \Set(\gamma)|\leq 1$. Thus, 
    \begin{enumerate}[(1)]
    \item $\pi=\alpha nn$, 
    \item $\pi=\alpha nbbn$ for $b\in[n-1]$, 
    \item $\pi=\alpha nnbb$, for $b\in[n-1]$, or
    \item $\pi=\alpha_1b\alpha_2nnb$ for $b\in[n-1]$. 
    \end{enumerate}
    Let $F_n = \{\pi \in \QQ_n(\Lambda) \mid \pi_{2n-1}=\pi_{2n}\}$ and $G_n = \{\pi \in \QQ_n(\Lambda) \mid \pi_{2n-1}\neq\pi_{2n}\}.$ Denote the sizes of these sets by $f_n=|F_n|$ and $g_n=|G_n|$. Then clearly, $F_n$ and $G_n$ are disjoint sets so that $F_n\cup G_n=\QQ_n(\Lambda)$. Thus, $\q_n(\Lambda) = f_n+g_n$. Notice that Cases (1) and (3) above are enumerated by $f_n$ and Cases (2) and (4) are enumerated by $g_n$. In Case (1), $\alpha$ can be any permutation in $\QQ_{n-1}(\Lambda)$ and in Case (3), $\alpha bb$ can be any permutation in $F_{n-1}$. In Case (2), $\alpha bb$ is any permutation in $F_{n-1}$ and in Case (4), $\alpha_1b\alpha_2b$ is any permutation in $\QQ_{n-1}(\Lambda)$. By observing these cases, we can see that $f_n=\q_{n-1}(\Lambda) + f_{n-1}$ and that $g_n=\q_{n-1}(\Lambda) + f_{n-1}$. In particular, $f_n=g_n$. 
    
    Taken together we see that $$\q_n(\Lambda) = f_n+g_n = 2(\q_{n-1}(\Lambda) + f_{n-1}) = 3\q_{n-1}(\Lambda).$$ This recurrence together with the initial condition $\q_2(\Lambda) = 4$ gives us the result.
%
    \end{proof}
	
	For a set $\Lambda \subseteq \S_3$, define $$Q_\Lambda(x) = \sum_{n\geq 1}\sum_{\pi\in\QQ_n(\Lambda)} x^n.$$
	This is exactly the generating function for $\QQ_n(\Lambda)$. We will use this notation in the next theorem.
\begin{theorem}
	\label{132, 213}
	Let $\Lambda = \{132, 213\}$. Then
	$$\displaystyle Q_\Lambda(x) = \frac{(1-x)^2}{1 - 3x + x^3}.$$
\end{theorem}
\begin{proof}
Let $\Lambda=\{132, 213\}$ and $n\geq 2$. Write $\pi\in\QQ_n(\Lambda)$ as $\pi = \alpha n \beta n \gamma$. Since $\pi$ avoids 213, $\alpha$ and $\beta$ are increasing. If $a\in \Set(\alpha)$ and $b\in\Set(\beta)$, then either $a>b$ in which case $abn$ is an occurrence of 213, or $a<b$ in which case $anb$ is an occurrence of 132. We cannot have that $a=b$ since $anbn$ would then be an occurrence of 2121. Thus, we must have that either $\Set(\alpha) =\varnothing$ or $\Set(\beta) = \varnothing$. Thus there are four possibilities:
\begin{enumerate}[(1)]
\item $\pi=nn\gamma$,
\item $\pi = n\beta n\gamma$ where $\Set(\beta) \neq \varnothing$,
\item $\pi=\alpha nn \gamma$ where $\Set(\alpha) \neq \varnothing$ and $\Set(\alpha)\cap\Set(\gamma)=\varnothing$, or
\item $\pi = \alpha n n\gamma$ where $\Set(\alpha)\cap\Set(\gamma)\neq\varnothing$.
\end{enumerate}
Notice that for any $a\in\alpha$, $b\in\beta$, and $c\in\gamma$, we must have that $a\geq c$ and $b>c$ since $\pi$ avoids 132.
	Let
	\begin{align*}
		F_n & = \set{\pi \in \qbar_n(\Lambda) \mid \pi = nn \gamma \text{ (Case (1))}},\\
		G_n & = \set{\pi \in \qbar_n(\Lambda) \mid \pi = n\beta n\gamma, \Set(\beta)\neq\varnothing \text{ (Case (2))}}, and\\
		H_n & = \set{\pi \in \qbar_n(\Lambda) \mid \pi = \alpha nn \gamma, \Set(\alpha)\neq\varnothing \text{ (Case (3) and Case (4))}}. 
	\end{align*}
	We define $f_n = \abs{F_n}$,\ \,$g_n = \abs{G_n}$, and $h_n = \abs{H_n}$.
	Since $F_n$, $G_n$, and $H_n$ are disjoint sets with $F_n\cup G_n\cup H_n = \QQ_n(\Lambda)$, we have $\q_n(\Lambda) = f_n + g_n + h_n$.
	Let us consider each case.
	\begin{itemize}\itemsep=0em
	\item For all $n\geq 2$, we have that $\gamma$ is any permutation in $\QQ_{n-1}(\Lambda)$. Thus we have that $f_n = \q_{n-1}(\Lambda)$.
	\item For all $n\geq 2$, we have that $\pi=n(i+1)(i+1)\ldots (n-1)(n-1)n\gamma$ where $\gamma \in\QQ_{i}(\Lambda)$. Thus we have that
	\begin{align*}
		g_n = \sum_{k=0}^{n-2}\q_{k}(\Lambda).
	\end{align*}
	\item Finally, for all $n\geq 2$, $\alpha$ is increasing and $a\geq c$ for all $a\in\alpha$ and $c\in\gamma$. Thus we must have that $\alpha$ ends in $(n-1)$. If $\Set(\alpha)\cap\Set(\gamma)=\varnothing$, then $\alpha$ ends in $(n-1)(n-1)$ and thus $\alpha\gamma$ is a permutation in $F_n$ or $H_n$. If $\Set(\alpha)\cap\Set(\gamma)\neq\varnothing$, then we have $|\Set(\alpha)\cap\Set(\gamma)|=1$ since $\pi$ avoids 132. Given that $\Set(\alpha)\cap\Set(\gamma=\{a\}$, since $\pi$ avoids 132, we must have that $a=\min(\alpha)$ and $a=\max(\gamma)$. writing $\pi=\alpha_1a\alpha_2nn\gamma_1a\gamma_2$, we must have $\Set(\alpha_1)=\varnothing$ and $\Set(\gamma_1)=\varnothing$. Further, $\alpha_2$ is increasing and $\gamma_2$ is any permutation in $\QQ_{a-1}(\Lambda)$. 
	Therefore
	 $$h_n = f_{n-1} +
	h_{n-1} +  \sum_{k=0}^{n-2}\q_{k}(\Lambda)=  f_{n-1} +
	h_{n-1} +  \q_{n-2}(\Lambda) + g_{n-1}= \q_{n-2}(\Lambda) + \q_{n-1}(\Lambda).$$
	\end{itemize}
	Thus we obtain the recurrence $$\q_{n}(\Lambda)= \q_{n-1}(\Lambda) + \sum_{k=0}^{n-2} \q_{k}(\Lambda)+
	 \q_{n-2}(\Lambda) + \q_{n-1}(\Lambda) =  \q_{n-1}(\Lambda) + \q_{n-2}(\Lambda) + \sum_{k=0}^{n-1}\q_k(\Lambda).$$
	Solving for the generating function $Q_\Lambda(x)$, where $Q_\Lambda(x) = \sum_{n=0}^\infty \q_{n}(\Lambda)x^n$, we then obtain the result from this recurrence.
\end{proof}

    \begin{theorem} 
    \label{123, 312}
For $n \geq 2$, $\q_n(132,321) = 2n^2 - 3n + 2$.
\end{theorem}
\begin{proof}
Let $\Lambda=\{132,321\}$ and $n>2$. Write $\pi \in \QQ_n(123, 312)$ as $\pi = \alpha n \beta n \gamma$. 
Since $\pi$ avoids $132$, we must have that for any $a\in\alpha$, $b\in\beta$, and $c\in\gamma$, $a>b$, $b>c,$ and $a\geq c$. Since $\pi$ avoids 321, it must be the case that at least one of $\Set(\beta)$ and $\Set(\gamma)$ is empty. Thus there are four possibilities:
\begin{enumerate}[(1)]
\item $\pi=\alpha nn$,
\item $\pi = \alpha n \beta n$ where $\Set(\beta)\neq \varnothing$,
\item $\pi=\alpha nn \gamma$ where $\Set(\gamma)\neq \varnothing$ and $\Set(\alpha)\cap\Set(\gamma) =\varnothing$, or
\item $\pi=\alpha nn \gamma$ where $\Set(\alpha)\cap\Set(\gamma) \neq\varnothing$. 
\end{enumerate}
In Case (1), $\alpha$ is any permutation in $\QQ_n(\Lambda)$. In Case (2), since $\pi$ avoid 321, $\beta$ must be increasing and since $\Set(\beta)\neq \varnothing$, $\alpha$ is also increasing. Thus there are $n-1$ permutations in this case. In Case (3), we similarly have that $\alpha$ and $\gamma$ are increasing and there are $n-1$ permutations in this case.

Finally, consider Case (4). Since $\pi$ avoids 132, we must have that $|\Set(\alpha)\cap\Set(\gamma)|= 1$, $a=\min(\Set(\alpha))$, and $a=\max(\Set(\gamma))$. Let us write $\pi = \alpha_1 a \alpha_2 nn\gamma_1a\gamma_2.$ Then $\alpha_1$, $\alpha_2$, and $\gamma_1$ are increasing and $\Set(\gamma_2) = \varnothing$. Further, since $\pi$ avoids 321, at least one of $\Set(\alpha_1)$ and $\Set(\gamma_1)$ are empty. There is one way for both $\Set(\alpha_1)$ and $\Set(\gamma_1)$ to be empty, there is $(n-2)$ permutations of the form $\alpha_1a\alpha_2nna$ with $\Set(\alpha_1)\neq \varnothing$, and there is $(n-2)$ permutations of the form $a\alpha_2nn\gamma_1a$ with $\Set(\gamma_1)\neq\varnothing$.
	
	Then we obtain the recurrence: 
	$$\q_{n}(\Lambda) = \q_{n-1}(\Lambda) + (n-1)+(n-1)+1+(n-2)+(n-2)
	 =\q_{n-1}(\Lambda) +4n-5 .$$ 
	 Solving this recurrence, we obtain the result.
	\end{proof}

\subsection{Avoiding three patterns in $\S_3$}

\begin{theorem}
\label{123, 132, 321}
For all $n \geq 5$, $\q_n(123,132, 321) =  0$.
\end{theorem}
\begin{proof}
    This follows immediately from Lemma \ref{no123and321}.
    \end{proof}

    \begin{theorem}
\label{132, 213, 321}
For all $n \geq 2$, $\q_n(132,213,321) =  2n$.
\end{theorem}

\begin{proof}
Let $\Lambda = \{132,213,321\}$. For any $\pi \in \QQ_n(\Lambda)$, write $\pi$ as $\pi=\alpha n\beta n\gamma$. Since $\pi$ avoids 213, $\alpha$ must be increasing. Since $\pi$ avoids 321, both $\beta$ and $\gamma$ are increasing as well. In addition, for any $a\in\Set(\alpha)$, $b\in\Set(\beta)$, and $c\in\Set(\gamma)$, $a> b$, $b> c$, and $a\geq c$ since $\pi$ avoids 132. For $n\geq 3$, we must also have $a>c$ since $\pi$ avoids 213 and 132. Notice that since $\pi$ avoids 213 and $a>b$ for any $a\in\Set(\alpha)$ and $b\in\Set(\beta)$, we must actually have that either $\Set(\alpha) = \varnothing$ or $\Set(\beta) = \varnothing$. 

Thus either $\{\Set(\alpha), \Set(\gamma)\}$ or $\{\Set(\beta), \Set(\gamma)\}$ is a partition of $[n-1]$ into two (possibly empty) disjoint sets of consecutive numbers. There are exactly $n$ ways to do this, thus there are $2n$ permutations in $\QQ_n(\Lambda)$.
\end{proof}

\begin{example}
For an example of the process in the proof of the previous theorem, suppose the partition of $[8]$ into two subsets is given by $$\{\{1,2,3\}, \{4,5,6,7,8\}\}.$$ Then either $\Set(\alpha) =\{4, 5,6,7,8\}$ and $\Set(\gamma) = \{1,2,3\}$ in which case  $\pi = 445566778899112233,$ or $\Set(\beta) =\{4, 5,6,7,8\}$ and $\Set(\gamma) = \{1,2,3\}$ in which case  $\pi = 994455667788112233.$
\end{example}

\

    \begin{theorem}
\label{123, 213, 312}
For all $n \geq 2$, $\q_n(123,213,312) =  2n$.
\end{theorem}
\begin{proof}
Let $\Lambda = \{123, 213, 312\}$ and $n\geq 2$. Suppose $\pi\in\QQ_n(\Lambda)$ with $\pi = \alpha n \beta n \gamma$. Since $\pi$ avoids 312, we must have that $\gamma$ is decreasing. Further, since $\pi$ avoids 123 and 213, we must have that $|\Set(\alpha)\cup \Set(\beta)|\leq 1$. Thus, there are four possibilities (all of which include $\gamma$ as a decreasing segment):
\begin{enumerate}[(1)]
\item $\pi = aann\gamma$ where $a\in[n-1]$, 
\item $\pi = nbbn\gamma$ where $b\in[n-1]$, 
\item $\pi = nn\gamma$, or 
\item $\pi = (n-1)nn(n-1)\gamma$.
\end{enumerate}
Since $\gamma$ is always determined, there are clearly $2n$ total permutations.
\end{proof}

    \begin{theorem}
\label{132, 213, 312}
For all $n \geq 2$, $\q_n(132,213,312) =  2n$.
\end{theorem}
    
\begin{proof}
	This proof is very similar to the proof of Theorem \ref{132, 213, 321}. Let $\Lambda = \{132,213,312\}$ and $n\geq 2$. Write $\pi \in \qbar_{n}(\Lambda)$ as $\pi = \alpha n \beta n \gamma$.

	Since $\pi$ avoids 213, $\alpha$ must be increasing, and since $\pi$ avoids 312, both $\beta$ and $\gamma$ are decreasing. In addition, for any $a\in\Set(\alpha)$, $b\in\Set(\beta)$, and $c\in\Set(\gamma)$, $a> b$, $b> c$, and $a\geq c$ since $\pi$ avoids 132. For $n\geq 3$, either $\Set(\beta) \neq \varnothing$, in which $a>c$ since $\pi$ avoids 132 and 213, or $\Set(\beta) = \varnothing$ in which we must also have $a>c$ since $\pi$ avoids 213 and 132. Notice that since $\pi$ avoids 213 and $a>b$ for any $a\in\Set(\alpha)$ and $b\in\Set(\beta)$, we must actually have that either $\Set(\alpha) = \varnothing$ or $\Set(\beta) = \varnothing$. 

Thus either $\{\Set(\alpha), \Set(\gamma)\}$ or $\{\Set(\beta), \Set(\gamma)\}$ is a partition of $[n-1]$ into two (possibly empty) disjoint sets of consecutive numbers. There are exactly $n$ ways to do this, thus there are $2n$ permutations.

\end{proof}
\begin{example}
For an example of the process in the proof of the previous theorem, suppose the partition of $[8]$ into two subsets is given by $$\{\{1,2,3\}, \{4, 5,6,7,8\}\}.$$ Then either $\Set(\alpha) =\{4,5,6,7,8\}$ and $\Set(\gamma) = \{1,2,3\}$, in which case  we obtain the permutation $\pi = 445566778899332211,$ or $\Set(\beta) =\{4,5,6,7,8\}$ and $\Set(\gamma) = \{1,2,3\}$, in which case  $\pi = 988776655449332211.$ 
\end{example}

\begin{theorem}\label{123, 132, 213}
	For all $n \geq 1$, $\q_n(123, 132, 213) =  \dfrac{(1+\sqrt{2})^{n+1}}{4} + \dfrac{(1-\sqrt{2})^{n+1}}{4} - \dfrac{(-1)^{n+1}}{2}$.
\end{theorem}
\begin{proof}
Let $\Lambda = \{123, 132, 213\}$ and $n>3$.  Write $\pi\in \QQ_n(\Lambda)$ as $\pi=\alpha n\beta n \gamma$. Since $\pi$ avoids 123 and 213, $|\Set(\alpha)\cup\Set(\beta)|\leq 1$. Since $\pi$ avoids 132, $\Set(\alpha)\cup\Set(\beta)=\varnothing$ or $\{n-1\}$. There are four cases:
\begin{enumerate}[(1)]
\item $\pi = nn\gamma$, 
\item $\pi = n(n-1)(n-1)n\gamma$, 
\item $\pi=(n-1)(n-1)nn\gamma$, or
\item $\pi=(n-1)nn\gamma_1(n-1)\gamma_2$. 
\end{enumerate}
Notice that in Case (4), since $\pi$ avoids 123 and 213, we have $|\Set(\gamma_1)|\leq 1$. Since $\pi$ also avoids 132, $\Set(\gamma_1)=\varnothing$ or $\{n-2\}$. 

Thus, there are $\q_{n-1}(\Lambda)$ permutations of Type (1), $\q_{n-2}(\Lambda)$ permutations of Type (2), $\q_{n-2}(\Lambda)$ permutations of Type (3), and $\q_{n-2}(\Lambda) + \q_{n-3}(\Lambda)$ permutations of Type (4). Thus, we obtain the recurrence $$\q_n(\Lambda) = \q_{n-1}(\Lambda) + 3\q_{n-2}(\Lambda) + \q_{n-3}(\Lambda).$$
By solving this recurrence  given the initial conditions that $\q_0(\Lambda) = 1, \q_1(\Lambda) = 1$, and $\q_2(\Lambda) = 4$, we obtain the result.
\end{proof}

    \begin{theorem}\label{123, 132, 312}
	For all $n \geq 1$, $\displaystyle \q_n(123, 132, 312) = \frac{1}{2}(n^2 + 3n - 2)$.
\end{theorem}	
\begin{proof}
Let $\Lambda=\{123, 132, 312\}$ and $n\geq2$. Write $\pi \in \QQ_n(\Lambda)$ as $\pi=\alpha n\beta n\gamma$. Then $\alpha$, $\beta$, and $\gamma$ are all decreasing since $\pi$ avoids $123$ and 312. In addition, since $\pi$ avoids $132$, we must have that for any $a\in\Set(\alpha)$, $b\in\Set(\beta)$, and $c\in\Set(\gamma)$, $a>b$, $b>c$, and $a\geq c$. 
We can only have $a=c$ if $\Set(\beta)=\varnothing$. This leaves two possibilities:
\begin{enumerate}[(1)]
\item $\Set(\alpha)\cap\Set(\gamma)=\varnothing$, or
\item $\Set(\alpha)\cap\Set(\gamma)\neq\varnothing$ and $\Set(\beta)=\varnothing$.
\end{enumerate}
In Case (1), $\{\Set(\alpha), \Set(\beta), \Set(\gamma)\}$ is a partition of $[n-1]$ into three (possibly empty) disjoint subsets of consecutive numbers. There are clearly ${{n+1}\choose 2}$ such partitions. 

In Case (2), we must have that $|\Set(\alpha)\cap\Set(\gamma)|=1$ since $\pi$ avoids both 123 and 312. Since both $\alpha$ and $\gamma$ are decreasing and $a\geq c$ for all $a\in\Set(\alpha)$ and $c\in\Set(\gamma)$, we thus have $\alpha'anna\gamma'$, and in fact $\alpha aa\gamma$ is exactly the decreasing permutation in $\S_{n-1,n-1}$. Since $a\in[n-1]$ and $\alpha'$ and $\gamma'$ are determined (namely $\alpha=(n-1)(n-1)\ldots(a+1)(a+1)$ and $\beta=(a-1)(a-1)\ldots 2211$), there are $n-1$ such permutations in Case (2). Taking Cases (1) and (2) together, we obtain the result.
\end{proof}

\subsection{Avoiding four patterns in $\S_3$}

\begin{theorem} 
\label{123, 132, 213, 321 and more}
For all $n \geq 5$, we have that $\q_n(123,132, 213,321)$, $\q_n(123,132, 231,321)$, and $\q_n(123,132, 312,321)$ are all $0$.
\end{theorem}
\begin{proof}
    This follows immediately from Lemma \ref{no123and321}.
    \end{proof}
    
\begin{theorem}\label{123, 132, 213, 231} 
For all $n \geq 2$, $\q_{n}(123, 132, 213, 231) = 4$.
\begin{proof}
Let $\Lambda=\{123,132,213,231\}$ and suppose $n>2$. Then for any $\pi\in\QQ_n(\Lambda)$, we must have that $\pi=nn\alpha$ with $\alpha\in\QQ_{n-1}$. Indeed, if there were some $i<j<k$ so that $\pi_i\neq n$ and either $\pi_j=n$ or $\pi_k=n$, where $\pi_i, \pi_j,$ and $\pi_k$ are distinct, then $\pi_i\pi_j\pi_k$ would be an occurrence of a forbidden pattern. Thus $\q_{n}(\Lambda)=\q_{n-1}(\Lambda)$. Since $\QQ_{2}(\Lambda) = \{1122,1221,2211,2112\}$, we must have that $\q_{n}(\Lambda)=4$ for all $n>2$.
\end{proof}
\end{theorem}
\begin{theorem} \label{123, 132, 231, 312}
For all $n \geq 3$, $\q_{n}(123, 132, 231, 312) = 3$.
\begin{proof}
Let $\Lambda=\{123, 132, 231, 312\}$ and suppose $n>2$.  Let $\pi \in \QQ_{n}(\Lambda)$. Since $\pi$ avoids 132 and 231, $n$ must lie either at the beginning or end of the permutation. If $\pi_1=n$, then the values in $\{1,1,2,2,\ldots, n-1,n-1\}$ must appear in decreasing order since $\pi$ avoids 312. Similarly, if $\pi_{2n}=n$, then the values in $\{1,1,2,2,\ldots, n-1,n-1\}$ must appear in decreasing order since $\pi$ avoids 123. Thus the only elements of $\QQ_{n}(\Lambda)$ are $nn(n-1)(n-1)\cdots 2211$, $n(n-1)(n-1)\cdots 2211n$, and $(n-1)(n-1)\cdots 2211nn$.\end{proof}
\end{theorem}
\begin{theorem}  \label{132, 213, 231, 312}
For all $n \geq 3$, $\q_{n}(132, 213, 231, 312) = 2$.
\begin{proof}
For all $n\geq 3$, the only patterns of length 3 that are allowed in a permutation $\pi \in \QQ_{n}(132, 213, 231, 312)$ are 123 and 321. Thus only the increasing permutation $1122\ldots nn$ and the decreasing permutation $nn\ldots 2211$ are in $\QQ_{n}(132, 213, 231, 312)$.
\end{proof}
\end{theorem}

\subsection{Avoiding five patterns in $\S_3$}

\begin{theorem} 
\label{123, 132, 213, 231, 321}
For all $n \geq 5$, $\q_n(123,132, 213,231,321)=0$.
\end{theorem}
\begin{proof}
    This follows immediately from Lemma \ref{no123and321}.
    \end{proof}
    
    \begin{theorem} 
\label{123, 132, 213, 231, 312}
For all $n \geq 3$, $\q_n(123,132, 213,231,312)=1$.
\end{theorem}
\begin{proof}
Let $n\geq 3$, and suppose $\pi \in\QQ_n(123,132, 213,231,312)$. Since the only pattern of length 3 that is allowed to appear is the decreasing pattern, every length 3 subsequence of $\pi$ is decreasing, thus $\pi$ is the decreasing permutation.
    \end{proof}

\subsection{Number of plateaus}

We finish this section by enumerating the set of quasi-Stirling permutations with a given number of plateaus, i.e.  \emph{consecutive} occurrences of the pattern 11. 
We note that the bijection $\varphi$, as defined in the previous section, sends leaves of an ordered rooted labeled tree to plateaus in the associated quasi-Stirling 
permutation. In other words, if the tree $T$ has a leaf labeled $i$, then there is some $j\in [2n-1]$ so that for $\pi=\varphi(T)$, $\pi_j=\pi_{j+1}=i$. Using a well-known result 
about the number of ordered rooted unlabeled trees on $[n]$ with $k$ leaves (see for example \cite{StanleyBook}), i.e. that there are $$\frac{1}{k} \binom{n-1}{k-1}\binom{n}{k-1}$$ of them, 
one can immediately obtain the following proposition. 

\begin{proposition}\label{cor:plateaus}
The number of quasi-Stirling permutations on $[n,n]$ with exactly $k$ plateaus is equal to $$ \frac{n!}{k} \binom{n-1}{k-1}\binom{n}{k-1}.$$
\end{proposition}

\section{Open questions}\label{sec: open}

\paragraph{Avoid other patterns.} In this paper, we do not enumerate the number of quasi-Stirling permutations that avoid a single pattern of length 3. Though we made some progress, we were unable to produce a recurrence or closed form for these permutations. Thus it remains open to enumerate $\QQ_n(321)$ and $\QQ_n(312)$.

\paragraph{Statistics.}
In this paper, we enumerate quasi-Stilring permutations by the number of plateaus. It is an open question to enumerate these permutations by the number of descents, by the number of left-to-right maxima, or any other statistic. Based on numerical evidence, we conjecture that the number of quasi-Stirling permutations on $[n,n]$ with exactly $n-1$ descents is $(n+1)^{n-1}$. 
Since this number coincides with the number of rooted \emph{unordered} forests on $n$ nodes, this may indicate one can use the bijection $\varphi$ to rooted ordered trees in some way to prove this conjecture.

\paragraph{Generalizations.}
In several contexts, Stirling permutations have been generalized to permutations of any multiset that avoid 212. A similar generalization could be done for quasi-Stirling permutations. It may be interesting to investigate such permutations.

\paragraph{Pattern avoidance in forests.} Pattern-avoidance is closely related to the study of pattern avoidance in forests. For example, the number of inversions (i.e, an occurrence of 21) in a rooted labeled ordered tree is equal to the number of occurrences of the pattern 2112 in quasi-Stirling permutations. The authors know of no study of pattern avoidance of ordered trees, but much has been studied for inversions of unordered trees (see for example, \cite{MR68, GSY95}) and other statistics and pattern avoidance in unordered trees (see for example \cite{GS06, AA18}).

\subsection*{Acknowledgements} The student authors on this paper, Adam Gregory, Bryan Pennington, and Stephanie Slayden, were funded as part of an REU at the University of Texas at Tyler sponsored by the NSF Grant DMS-1659221.

\bibliographystyle{plain}
\bibliography{stirlingrefs}

\begin{thebibliography}{10}

\bibitem{AA18}
K.~Archer and K.~Anders.
\newblock Rooted forests that avoid sets of permutations.
\newblock Preprint at arXiv:1607.03046.

\bibitem{B2008}
M.~B\'{o}na.
\newblock Real zeros and normal distribution for statistics on stirling
  permutations defined by gessel and stanley.
\newblock {\em SIAM J. Discrete Math.}, 23:401--406, 2008.

\bibitem{CMM2016}
D.~Callan, S.~Ma, and T.~Mansour.
\newblock Restricted stirling permutations.
\newblock {\em Taiwanese J. Math.}, 20(5):957--978, 2016.

\bibitem{DY2014}
A.~Dzhumadil'daev and D.~Yeliussizov.
\newblock Stirling permutations on multisets.
\newblock {\em Europ. J. Combin.}, 36:377--392, 2014.

\bibitem{GSY95}
I.~Gessel, B.~Sagan, and Y.-N. Yeh.
\newblock Enumeration of trees by inversions.
\newblock {\em J. Graph Theory}, 19(4):435--459, 1995.

\bibitem{GS06}
I.~Gessel and S.~Seo.
\newblock A refinement of cayley's formula for trees.
\newblock {\em Electron. J Combin.}, 11(2):\#R27, 2006.

\bibitem{GS1978}
I.~Gessel and R.~Stanley.
\newblock Stirling polynomials.
\newblock {\em J. Combinatorial Theory, Ser. A}, 24:24--33, 1978.

\bibitem{J2008}
S.~Janson.
\newblock Plane recursive trees, stirling permutations and an urn model.
\newblock {\em Discrete Mathematics and Theoretical Computer Science},
  Proceedings AI:541--547, 2008.

\bibitem{JKP2011}
S.~Janson, M.~Kuba, and A.~Panholzer.
\newblock Generalized stirling permutations, families of increasing trees and
  urn models.
\newblock {\em J. Combinatorial Theory, Ser. A}, 118(1):94--114, 2011.

\bibitem{Kitaevbook}
S.~Kitaev.
\newblock {\em Patterns in Permutations and Words}.
\newblock Springer, 2011.

\bibitem{KP2011}
M.~Kuba and A.~Panholzer.
\newblock Analysis of statistics for generalized stirling permutations.
\newblock {\em Combinatorics, Probability, and Computing}, 20:875--910, 2011.

\bibitem{KP2012}
M.~Kuba and A.~Panholzer.
\newblock Enumeration formulae for pattern restricted stirling permutations.
\newblock {\em Discrete Math.}, 312(21):3179--3194, 2012.

\bibitem{MY2015}
S.~Ma and Y.~Yeh.
\newblock Stirling permutations, cycle structure of permutations and perfect
  matchings.
\newblock {\em Electron. J Combin.}, 22(4):\#P4.42, 2015.

\bibitem{MR68}
C.L. Mallows and J.~Riordan.
\newblock The inversion enumerator for labeled trees.
\newblock {\em Bull. Amer. Math. Soc.}, 74(1):92--94, 1968.

\bibitem{JW2015}
J.B. Remmel and A.~Wilson.
\newblock Block patterns in stirling permutations.
\newblock {\em J. Comb.}, 6(1-2):179--204, 2015.

\bibitem{SS1985}
R.~Simion and F.~Schmidt.
\newblock Restricted permutations.
\newblock {\em Europ. J. Combin.}, 6:383--406, 1985.

\bibitem{StanleyBook}
R.~Stanley.
\newblock {\em Enumerative Combinatorics: Volume 1}.
\newblock Cambridge University Press, New York, NY, USA, 2nd edition, 2011.

\end{thebibliography}

\end{document}